\documentclass[twoside,12pt,reqno]{amsart}
\usepackage[usenames,dvipsnames]{color}
\usepackage{amscd, bbm, amsaddr, mathtools}
\usepackage{amssymb,mathabx}
\usepackage{todonotes}
\usepackage{mathrsfs,wasysym, tikz, tikz-cd, extpfeil}
\usetikzlibrary{matrix, arrows}
\usepackage{amsmath}
\usepackage{mathrsfs}
 \usepackage[margin= 30mm]{geometry}
\usepackage{hyperref}

\hypersetup{colorlinks=true, linkcolor = Blue, citecolor=Purple}

\makeatletter
\@addtoreset{equation}{section}

\numberwithin{equation}{section}

\newtheorem{Theorem}{Theorem}[section]
\newtheorem*{Theorem*}{Theorem}
\newtheorem*{Corollary*}{Corollary}

\newtheorem{Lemma}[Theorem]{Lemma}
\newtheorem{Proposition}[Theorem]{Proposition}
\newtheorem{Corollary}[Theorem]{Corollary}

\theoremstyle{definition}

\theoremstyle{remark}
\newtheorem{Remark}[Theorem]{Remark}

\newtheorem*{Remark*}{Remark}

\newtheorem*{proofofKW1}{Proof of Main Theorem}

\newtheorem{Example}[Theorem]{Example}

\newbox\squ  
\setbox\squ=\hbox{\vrule width.6pt
   \vbox{\hrule height.6pt width.4em\kern1ex
         \hrule height.6pt}%
   \vrule width.6pt}

\newcommand{\C}{\mathbb{C}}
\newcommand{\N}{\mathbb{N}}
\newcommand{\Z}{\mathbb{Z}}

\newcommand{\g}{\mathfrak{g}}
\newcommand{\gl}{\mathfrak{gl}}
\renewcommand{\b}{\mathfrak{b}}
\newcommand{\h}{\mathfrak{h}}

\newcommand{\p}{\mathfrak{p}}
\newcommand{\s}{\mathfrak{s}}

\renewcommand{\O}{\mathcal{O}}

\newcommand{\ad}{\operatorname{ad}}

\newcommand{\Lie}{\operatorname{Lie}}

\newcommand{\Hom}{\operatorname{Hom}}
\newcommand{\GL}{\operatorname{GL}}

\newcommand{\gr}{\operatorname{gr}}

\newcommand{\Spec}{\operatorname{Spec}}

\newcommand{\reg}{{\operatorname{reg}}}

\newcommand{\End}{\operatorname{End}}

\newcommand{\Mat}{\operatorname{Mat}}

\newcommand{\Ind}{\operatorname{Ind}}

\newcommand{\Der}{\operatorname{Der}}

\renewcommand{\mod}{\operatorname{-mod}}

\newcommand{\Frac}{\operatorname{Frac}}

\newcommand{\ind}{\operatorname{ind}}

\newcommand{\Char}{{\operatorname{char}}}
\newcommand{\alg}{\operatorname{-alg}}

\newcommand{\bs}{\bar \s}

\newcommand{\Lr}{\mathcal{L}_{\operatorname{ring}}}

\title[]{\boldmath  A proof of the first Kac--Weisfeiler conjecture \\ in large characteristics}

\author{Benjamin Martin, David Stewart and Lewis Topley \medskip \newline
with an appendix by Akaki Tikaradze}


\begin{document}

\maketitle

\begin{abstract}
In 1971, Kac and Weisfeiler made two influential conjectures describing the dimensions of simple modules of a restricted Lie algebra $\g$. The first predicts the maximal dimension of simple $\g$-modules and in this paper we apply the Lefschetz Principle and classical techniques from Lie theory to prove this conjecture for all restricted Lie subalgebras of $\gl_n(k)$ whenever $k$ is an algebraically closed field of characteristic $p \gg 0$. As a consequence we deduce that the conjecture holds for the the Lie algebra of a group scheme over any commutative ring, after specialising to an algebraically closed field of almost any characteristic.

In the appendix to this paper, written by Akaki Tikaradze, an alternative, short proof of the first Kac--Weisfeiler conjecture is given for the Lie algebra of group scheme over a finitely generated ring $R \subseteq \C$, after base change to a field of large positive characteristic.
\end{abstract}

\section{Introduction}

Since the pioneering work of Zassenhaus \cite{Zas54}, it has been known that the dimensions of simple modules of finite-dimensional Lie algebras over a field $k$ of characteristic $p > 0$ are bounded, and that the maximal dimension, which we denote $M(\g)$, is a power of $p$. Jacobson introduced the notion of a restricted Lie algebra with a view to developing a Galois theory for purely inseparable field extensions \cite{Jac37}. Very briefly, restricted Lie algebras are those that admit a $p$-power map $x \mapsto x^{[p]}$ satisfying axioms which are modelled on the properties of the map $\Der_k(A)  \rightarrow  \Der_k(A)$ given by $d  \mapsto  d^p$, where $A$ is an associative $k$-algebra. Many of the modular Lie algebras arising in nature are restricted: for example, when $\g$ is the Lie algebra of an algebraic group scheme $G$ of finite type over $k$ then there is a natural $G$-equivariant restricted structure on $\g$ \cite[I.7.10]{Jan03}.

Now let $k$ be algebraically closed. In \cite{KW71} Kac and Weisfeiler carried out the first intensive study of representations of restricted Lie algebras. The key property of the restricted structure on $\g$ is that the elements $x^p - x^{[p]}$ are central in $U(\g)$ for $x \in \g$, and the subalgebra $Z_p(\g) \subseteq U(\g)$ generated by these elements is known as the \emph{$p$-centre}. One of the fundamental insights of \cite{KW71} is that the maximal ideals of $Z_p(\g)$ are parametrised by $\g^*$. Since the universal enveloping algebra $U(\g)$ is finite over its $p$-centre it follows from Hilbert's Nullstellensatz that every simple $\g$-module is annihilated by a unique maximal ideal of $Z_p(\g)$, and so to every simple $\g$-module $M$ we may assign some linear form $\chi \in \g^*$ known as the {\it $p$-character of $M$}. This situation is reminiscent of Kirillov's orbit method, and so it is natural to hope that global properties of the module category $\g\mod$ will be controlled by geometric properties of the module $\g^*$. These aspirations were formalised by Kac--Weisfeiler in the form of two conjectures: the first of these predicts the maximal dimension of simple $\g$-modules, and in the current paper we apply techniques from model theory to confirm that conjecture for all restricted Lie subalgebras of $\gl_n(k)$ when the characteristic of the field $k$ is large. The second conjecture proposes lower bounds on powers of $p$ dividing the dimensions of $\g$-modules with $p$-character $\chi$; for more detail see \cite{PS99} and the references therein.

The coadjoint stabiliser of $\chi \in \g^*$ is denoted $\g^\chi$ and the {\it index of $\g$} is defined by
\begin{eqnarray}
\ind(\g) := \min_{\chi \in \g^*} \dim \g^\chi.
\end{eqnarray}
The first Kac--Weisfeiler conjecture (KW1) predicts that when $\g$ is any restricted Lie algebra the maximal dimension of simple $\g$-modules is
\begin{eqnarray}
\label{e:stateKW1}
M(\g) = p^{\frac{1}{2}(\dim \g - \ind \g)}.
\end{eqnarray}

\begin{Theorem}
\label{mainthm}
For all $d \in \N$ there exists $p_0 \in \N$ such that if $k = \overline{k}$ is a field of characteristic $p > p_0$ and $\g \subseteq \gl_d(k)$ is a restricted Lie subalgebra, then the first Kac--Weisfeiler conjecture holds for $\g$.
\end{Theorem}
In future work we shall provide an explicit bound on $p_0$ in the theorem. We now outline the proof of the result stated. In \cite{PS99} Premet and Skryabin studied deformations of reduced enveloping algebras to spectacular effect: one of their many results states that $M(\g) \ge p^{\frac{1}{2}(\dim\g - \ind \g)}$ for any restricted Lie algebra $\g$ and so it remains to prove the opposite inequality. 
In Kirillov's thesis he introduced the notion of a {\it polarisation of a linear form $\chi$}, which is a Lie subalgebra $\s\subseteq \g$ satisfying $\chi[\s,\s] =0$ and $\dim \s = \frac{1}{2}(\dim \g + \dim \g^\chi)$. These turn out to be central to the classification of primitive ideals in enveloping algebras of complex solvable Lie algebras \cite[Ch.\ 6]{Dix96}, as well as the classification of simple modules over restricted solvable Lie algebras \cite[\textsection 2]{KW71}. We say that $\s$ is a {\it weak polarisation of $\chi \in \g^*$} if $\s \subseteq \g$ is a Lie subalgebra of dimension $\frac{1}{2}(\dim \g + \ind \g)$ satisfying $\chi[\s,\s] = 0$. Solvable weak polarisations are known to exist for every linear form on every finite-dimensional complex Lie algebra, and we deduce that the same holds for all modular Lie algebras in large characteristics.  To do this we apply a first-order version of the Lefschetz Principle.  The condition for a solvable weak polarisation to exist involves polynomials in the structure constants of $\g$; we need to express this condition in terms of first-order sentences, and the key idea is to quantify over all Lie algebras of a fixed dimension simultaneously.
The proof concludes by observing that every simple module is a quotient of a module induced from a restricted Lie subalgebra containing a solvable polarisation (Theorem~\ref{C:chchchchchangesagain}). This places the required upper bound on the dimension of simple modules. After providing an elementary introduction to the Lefschetz Principle and the representation theory of restricted Lie algebras in \textsection \ref{s:prelims}, we give the proof of Theorem~\ref{mainthm} in \textsection \ref{S:KW1section}.

It is worth comparing the proof sketched above to the situation for Lie algebras of reductive groups. When $\g$ is such a Lie algebra it is known that for every $\chi \in \g^*$ there exists a Borel subalgebra $\b \subseteq \g$ such that $\chi[\b, \b] = 0$, $\dim \b = \frac{1}{2}(\dim \g + \ind \g)$ and $[\b, \b]$ is unipotent. It follows quickly that every simple $\g$-module of $p$-character $\chi$ is a quotient of some baby Verma module. These are defined to be the $U_\chi(\g)$-modules induced from one-dimensional $U_\chi(\b)$-modules. Hence the Borel subalgebras play the role of solvable weak polarisations in the reductive case.

Let $R$ be a commutative unital ring and say that $k$ is an $R$-field if $k$ is a field with an $R$-algebra structure. If $G$ is an $R$-group scheme and $k$ is an $R$-field then we write $G_k$ for the base change of $G$ from $R$ to $k$.  In our next theorem we describe a fruitful source of examples to which our main theorem can be applied; note that we do not need the group schemes $G_k$ to be reduced.
\begin{Theorem}
\label{thm:groupthm}
Let $G$ be an affine group scheme of finite type over $R$. There exists $p_0 \in \N$ such that when $p > p_0$ is prime and $k = \overline{k}$ is an $R$-field of characteristic  $p$, the first Kac--Weisfeiler conjecture holds for the Lie algebra $\Lie(G_k)$. 
\end{Theorem}
Thus for a fixed group scheme, the KW1 conjecture holds in {\it almost all} characteristics. The proof, which is presented in \textsection \ref{ss:groupschemes}, demonstrates that the Lie algebra $\Lie(G_k)$ admits a faithful restricted representation of dimension $d$ independent of the choice of characteristic $p > 0$ of the field $k$, which allows us to apply Theorem~\ref{mainthm}.

Until recently it was considered to be possible that \eqref{e:stateKW1} might hold for non-restricted Lie algebras; however, counterexamples to this hope were found by the third author, by presenting pairs of Lie algebras with isomorphic enveloping algebras and distinct indexes \cite{To17}.  We give some new examples in \textsection\ref{sec:notres}.  We produce a family of Lie algebras parameterised by the primes $p$ with the property that $\g_p$ is restricted if and only if \eqref{e:stateKW1} holds, which is the case if and only if $p \equiv 1$ modulo 4 (see Proposition~\ref{prop:notres}).

 \medskip

\noindent {\bf Acknowledgements:} The authors would like to thank Akaki Tikaradze for useful correspondence and for contributing the appendix to this paper. The third author also gratefully acknowledges the support of EPSRC grant number EP/N034449/1.

\section{Preliminaries}
\label{s:prelims}

\subsection{Model theory and the Lefschetz Principle}
\label{ss:modeltheory}
Since the main results and motivations of this paper come from representation theory, we expect that some of the readers will be unfamiliar with the model-theoretic methods which we use at several points. Hence we include here a short recap of some of the elements of model theory; a more detailed introduction to the theory may be found in \cite{Ma02}. Since our goal is to explain the Lefschetz Principle, we work exclusively with the language of rings.

The \emph{language of rings $\Lr$} consists of the binary function symbols symbols $+, -, \times$ and the constant symbols $0, 1$ \cite[Example~1.2.8]{Ma02}.  We consider the collection of first-order formulas in this language: these are the formulas that can be built from the symbols $\{\forall, \exists, \vee, \wedge, \neg, +, -, \times, 0, 1, =\}$ along with arbitrary choice of variables (see \cite[Section~1.1]{Ma02} for a precise definition of first-order formula). For example, for $n > 0$ fixed the following are formulas in $\Lr$:
\begin{eqnarray}
\label{e:notasentence}
& & (\forall x)(\forall y) (x^n + y^n \neq z^n);\\
\label{e:inverses}
& & (\forall x) (\exists y)\left( (xy = 1)) \vee (x = 0)\right); \\
\label{e:algclosed}
& & (\forall x_0) (\forall x_1) \cdots (\forall x_{n-1}) (\exists y) (y^n + x_{n-1} y^{n-1} + \cdots + x_0 = 0).
\end{eqnarray}
We say that a formula is a \emph{sentence} if every variable is bound to a quantifier; for example the formula \eqref{e:notasentence} is not a sentence because $z$ is a free variable, whilst \eqref{e:inverses} and \eqref{e:algclosed}
are both sentences in $\Lr$. If $\phi$ is a formula with free variables $x_1,...,x_n$ then we might indicate this by writing $\phi = \phi(x_1,...,x_n)$. In this case we can obtain a sentence from $\phi$ by binding the free variables to quantifiers. For example, if $\phi = \phi(z)$ is the formula from \eqref{e:notasentence} then $(\forall z) \phi(z)$ is a sentence in the language of rings. In this way we may use formulas to build sentences.

For $p \geq 0$ we record one more first-order sentence $\psi_p$ in $\Lr$:
\begin{eqnarray}
\label{e:charisp}
\psi_p & : & \underbrace{1 + \cdots + 1}_{p \text{ times}} = 0.
\end{eqnarray}

An \emph{$\Lr$-structure} is a set $R$ together with elements $0_R, 1_R \in R$, binary operations $+_R, -_R, \times_R: R \times R \rightarrow R$, and the binary relation $=_R$ which is always taken to be the diagonal embedding $R \subseteq R\times R$. For example, every ring $R$ gives rise to an $\Lr$-structure in the obvious way.

Later in this paper we will need to express some statements about elements of vector spaces as formulas and sentences in the language $\Lr$. To prepare for those arguments we now record a few examples which illustrate this procedure.
\begin{Example}
\label{E:firstordersentences}
Let $k$ be a field. The following statements can be formulated as sentences and formulas in $\Lr$ using appropriate variables:
\begin{enumerate}
\item there exist elements $x_1,...,x_m \in k^n$ which are linearly independent in $k^n$;
\item a given bilinear map $f : k^n\times k^n \rightarrow k^n$ defines a Lie bracket $[\cdot, \cdot]_f$ on $k^n$;
\item the Lie algebra $(k^n, [\cdot, \cdot]_f)$ is solvable.
\end{enumerate}
Moreover, the sentences and formulas we obtain do not depend on the choice of $k$.  To see this, we view $k^n$ as the set of tuples $(a_1,...,a_n)$ of elements of $k$. Then (1) is equivalent to the following first-order sentence
\begin{eqnarray*}
& & (\forall (i,j)\in \{1,\dots n\}\times \{1,\dots,m\})(\exists a_{i,j})(\forall b_1,...,b_m) \\ 
& & ((\sum_{j = 1}^m  b_j a_{1,j} = \sum_{j = 1}^m = b_j a_{2 ,j} = \cdots = \sum_{j = 1}^m b_j a_{m,j} = 0)
\Rightarrow (b_1 = b_2 = \cdots = b_m = 0)),
\end{eqnarray*}
where we take $x_j= (a_{1,j},\ldots, a_{n,j})$ for $1\leq j\leq m$.  This sentence does not depend on $k$.

Let $v_1,...,v_n \in k^n$ denote the standard basis. A bilinear map $f : k^n \times k^n \rightarrow k^n$ satisfies $f(v_i, v_j) = \sum_{l=1}^n f_{i,j;l} v_k$ for some scalars $f_{i,j;l}$, and we identify $f$ with the array $(f_{i,j;l})_{1\leq i,j,l \leq n} \in k^{n^3}$. Now the claim that \emph{$f$ defines a Lie bracket} can be encoded as a collection of linear and quadratic polynomial relations not depending on $k$ with integral coefficients in the variables
\begin{eqnarray}
\label{e:mapvariables}
(f_{i,j;l})_{1\leq i,j,l\leq n}.
\end{eqnarray}
These relations correspond to skew-symmetry and the Jacobi identity. Since every integer can be constructed using only the symbols $\{+, -, 1, 0\}$ it follows that \emph{$f$ defines a Lie bracket} is a first-order formula not depending on $k$ in $\Lr$ with free variables \eqref{e:mapvariables}. Similarly the statement \emph{the Lie algebra $(k^n, [\cdot, \cdot]_f)$ is solvable} can be encoded in terms of the vanishing of all $n$-fold iterations of the Lie bracket, which is equivalent to the vanishing of a collection of homogeneous polynomials of degree $2^n-1$ amongst the variables \eqref{e:mapvariables}. Again this is a first-order formula not depending on $k$ in $\Lr$ with free variables \eqref{e:mapvariables}.
\end{Example}

An \emph{$\Lr$-theory} is a set $T$ of first-order sentences in $\Lr$. If $\phi$ is a sentence and $M := (R, +_R, -_R, \times_R, 0_R, 1_R, =_R)$ is an $\Lr$-structure then we say that \emph{$M$ is a model of $\phi$}, and write $M \vDash \phi$, if the sentence $\phi$ is true when interpreted in $M$. If $T$ is an $\Lr$-theory then we say that \emph{$M$ is a model of $T$} and write $M \vDash T$ if $M \vDash \phi$ for all $\phi \in T$. 

One way to obtain a theory is to take the collection of all sentences which are true for every model of a particular class of mathematical object. Since mathematical objects are usually determined by axioms, we shall briefly explain (by way of an example) how to pass from a set of axioms to a theory of this type. Consider the set $A$ of axioms of commutative rings, which are clearly first-order sentences in $\Lr$. We may then consider the set ${\sf CR} \subseteq \Lr$ of sentences that are true for every model of $A$: i.e., those that are true in every commutative ring. Thus {\sf CR} denotes the theory of commutative rings. To illustrate the notation introduced above, observe that $(R, +_R, -_R, \times_R, 0_R, 1_R, =_R) \vDash {\sf CR}$ is equivalent to the statement that $(R, +_R, -_R, \times_R, 0_R, 1_R)$ is a commutative ring. As such we may slightly abuse terminology and identify the class of models of ${\sf CR}$ with the class of commutative rings.

In this paper we will be primarily interested in the theory of fields. The axioms of a field can obviously be written as first-order sentences in $\Lr$; for instance \eqref{e:inverses} expresses the existence of multiplicative inverses. The axioms of the algebraically closed fields are obtained by including the sentences \eqref{e:algclosed} for all $n > 0$. If $p > 0$ is prime then we may include the sentence $\psi_p$, defined in \eqref{e:charisp}, to obtain the axioms of the algebraically closed fields of characteristic $p > 0$; the corresponding theory is denoted {\sf ACF}$_p$. Alternatively we may include the sentences $\{ \neg \psi_p \mid p > 0\}$ to obtain the axioms of the algebraically closed fields of characteristic zero, and we denote their theory by {\sf ACF}$_0$ (cf.\ \cite[Example~1.2.8]{Ma02}). Since it will cause no confusion we identify the class of all models of {\sf ACF}$_p$ with the class of all algebraically closed fields of characteristic $p$, whenever $p \geq 0$ is fixed.

The following result is G\"odel's first completeness theorem in the context of $\Lr$ (cf.\ \cite[Theorem 2.1.2]{Ma02}).
\begin{Lemma}
\label{L:goedel}
Let $\phi$ be a first-order sentence and $T$ let be any theory in $\Lr$. Then $\phi$ is true when interpreted in every model of $T$ if and only if $\phi$ can be deduced from $T$ by means of a formal proof in $\Lr$.
\end{Lemma}
We say that an $\Lr$-theory $T$ is \emph{complete} if, for every first-order sentence $\phi$ in $\Lr$, either $\phi$ is true when interpreted in every model of $T$, or $\neg \phi$ is true when interpreted in every model of $T$. By Lemma~\ref{L:goedel} this is equivalent to saying that for every sentence $\phi$ we can derive either $\phi$ or $\neg \phi$ from $T$ by means of a formal proof. The following well-known result is proved by quantifier elimination \cite[Corollary~3.2.3]{Ma02}.
\begin{Theorem}
\label{T:ACcompleteness}
For $p = 0$ or $p$ prime, the theory ${\sf ACF}_p$ is complete.
\end{Theorem}

As an immediate consequence we obtain a first-order version of the Lefschetz Principle from algebraic geometry \cite[Corollary~2.2.10]{Ma02}:
\begin{Corollary} (Lefschetz Principle)
\label{C:lefschetz}
If $\phi$ is a sentence in $\Lr$ then:
\begin{enumerate}
\item If $\phi$ is true in some model of {\sf ACF}$_p$ where $p \ge 0$ then $\phi$ is true in every model of {\sf ACF}$_p$.
\item If $\phi$ is true in some model of {\sf ACF}$_0$ then there exists $p_0 \in \N$ such that $\phi$ is true in any model of ${\sf ACF}_p$ for $p > p_0$.
\end{enumerate}
\end{Corollary}
\begin{proof}
Part (1) is precisely Theorem~\ref{T:ACcompleteness}. For part (2) suppose that $\phi$ is true over some field satisfying the axioms of ${\sf ACF}_0$. Then by part (1) it is true for every such field, and by Lemma~\ref{L:goedel} we conclude that there exists a formal proof for $\phi$ in $\Lr$ using only the axioms of {\sf ACF}$_0$. Since the proof of $\phi$ involves only finitely many sentences in $\Lr$, it follows that the set of primes
$$P_\phi := \{p \mid \neg \psi_p \text{ occurs in the proof of } \phi\}$$
is finite, where $\psi_p$ is defined in \eqref{e:charisp}. Hence for $p > \max (P_\phi)$ there is a formal proof of $\phi$ using the axioms of ${\sf ACF}_p$. Using Lemma~\ref{L:goedel} once more we see that $\phi$ is true for every $\Lr$-structure satisfying the axioms of ${\sf ACF}_p$.
\end{proof}

\subsection{Restricted Lie algebras and reduced enveloping algebras}

Fix a field $k$ of characteristic $p > 0$ and let $\g$ be a Lie algebra over $k$. As usual we write $U(\g)$ for the enveloping algebra and $Z(\g)$ for the centre of $U(\g)$. Then $\g$ is said to be a \emph{restricted Lie algebra over $k$} if it comes equipped with a $p$-map $\g \to \g$, written $x \mapsto x^{[p]}$, which satisfies two axioms: if we write $\xi : \g \to U(\g)$ for the map $x \mapsto x^p - x^{[p]}$, then $(-)^{[p]}$ must satisfy
\begin{enumerate}
\item $\xi(\g) \subseteq Z(\g)$;
\smallskip
\item $\xi$ is $p$-semilinear in the sense of \cite[Lemma~2.1]{Jan98}.
\end{enumerate}
It follows from the PBW theorem that the vector space $\xi(\g)$ generates a polynomial algebra of rank equal to $\dim(\g)$ inside $U(\g)$. This algebra is referred to as the $p$-centre of $U(\g)$, and is denoted $Z_p(\g)$. If $\{x_i \mid i \in I\}$ is a basis for $\g$ then the PBW theorem for $U(\g)$ implies that $Z_p(\g)$ is isomorphic to a polynomial ring generated by $\{\xi(x_i) \mid i \in I\}$. Hence $Z_p(\g)$ can be naturally identified with the coordinate ring $k[(\g^*)^{(1)}]$ on the Frobenius twist of $\g^*$. 

When $k$ is an algebraically closed field we have $\g^* = (\g^*)^{(1)}$ as sets and so for every $\chi \in \g^*$ there is a maximal ideal $I_\chi \in \Spec Z_p(\g)$. Explicitly we have $I_\chi := (x^p - x^{[p]} - \chi(x)^p \mid x \in \g)$ and the \emph{reduced enveloping algebra with $p$-character $\chi$} is defined to be
$$U_\chi(\g) := U(\g) / U(\g)I_\chi.$$
We have $\dim U_\chi(\g)= p^{\dim \g}$, so a simple module for $U_\chi(\g)$ can have dimension at most $p^{\frac{1}{2}{\dim \g}}$.  If $\g_0 \subseteq \g$ is a restricted subalgebra and $\chi \in \g^*$ then we can abuse notation by identifying $U_{\chi|_{\g_0}}(\g_0)$ with the subalgebra of $U_\chi(\g)$ generated by $\g_0$, and denote this subalgebra by $U_\chi(\g_0)$. We say that a $\g$-module $M$ has $p$-character $\chi$ if the corresponding representation $U(\g) \rightarrow \End_k(M)$ factors through the quotient $U(\g) \rightarrow U_\chi(\g)$. If $\g_0 \subseteq \g$ are restricted Lie algebras, $\chi \in \g^*$ and $M_0$ is a $U_\chi(\g_0)$-module then we may define the induced module
\begin{eqnarray}
\Ind_{\g_0}^{\g, \chi}(M_0) := U_\chi(\g) \otimes_{U_\chi(\g_0)} M_0.
\end{eqnarray}
We have
\begin{eqnarray}
\label{e:induceddimn}
\dim \Ind_{\g_0}^{\g, \chi}(M_0) = p^{\dim \g - \dim \g_0} \dim(M_0)
\end{eqnarray}
and this induced module is universal amongst $U_\chi(\g)$-modules $M$ such that the restriction to $U_\chi(\g_0)$ contains a submodule isomorphic to $M_0$: that is, any nonzero homomorphism of $U_\chi(\g_0)$-modules $M_0 \rightarrow M$ induces a nonzero homomorphism $\Ind_{\g_0}^{\g, \chi}(M_0) \rightarrow M$ of $U_\chi(\g)$-modules.

The coadjoint $\g$-module is the vector space $\g^*$ with module structure given by $(x \cdot \chi)(y) := \chi([y,x])$ where $x, y\in \g$ and $\chi \in \g^*$. The stabiliser of $\chi \in \g^*$ is then defined to be
$$\g^\chi := \{x \in \g \mid x\cdot \chi = 0\} = \{x \in \g \mid \chi [x, \g] = 0\}$$
and the index of $\g$ --- commonly denoted $\ind(\g)$ --- is the minimal dimension of $\g^\chi$ as $\chi$ varies over all elements of $\g^*$.

\section{The maximal dimensions of simple modules}
\label{S:KW1section}

In this section we prove Theorem~\ref{mainthm}. In order to do so we recall a few pieces of terminology. If $\g$ is a Lie algebra over a field $k$ and $\chi \in \g^*$ we say that a Lie subalgebra $\s\subseteq \g$ is \emph{subordinate to $\chi$} if $\s$ is an isotropic subspace with respect to the skew-symmetric
form\begin{eqnarray}
\label{e:form}
\begin{array}{rcl} B_\chi : \g\times\g &\rightarrow & k \\ (x,y) & \mapsto & \chi([x,y]):\end{array}
\end{eqnarray}
in other words, if $\chi([\s,\s]) = 0$. We say that a subalgebra $\s\subseteq \g$ is a \emph{polarisation of $\chi$} if $\s$ is a Lagrangian for $B_\chi$, i.e., $\s$ is a maximal isotropic subspace of $\g$. Since the stabiliser $\g^\chi$ coincides with the radical of $B_\chi$ it follows from \cite[1.12.1]{Dix96}
that
\begin{eqnarray}
\label{e:polarisationdim}
\dim(\s) \leq \frac{1}{2}(\dim(\g) + \dim (\g^\chi))
\end{eqnarray}
if $\s$ is subordinate to $\chi$. Furthermore equality holds if and only if $\s$ is a polarisation of $\chi$. Finally we say that a Lie subalgebra $\s \subseteq \g$ is a {\it weak polarisation of $\chi$} if $\s$ is isotropic for the form \eqref{e:form} and
\begin{eqnarray}
\label{e:weakpolarisationdim}
\dim(\s) = \frac{1}{2}(\dim(\g) + \ind (\g)).
\end{eqnarray}

The proof of Theorem~\ref{mainthm} rests on the existence of solvable weak polarisations for linear forms, and the following result is the key step.
\begin{Proposition}
\label{P:solvablepolarisationsexist}
For all $n, d \in \N$, there exists $p_1 = p_1(n,d) \in \N$ such that if:
\begin{enumerate}
\item $k$ is an algebraically closed field of characteristic $p > p_1$;
\item $\g$ is a Lie algebra of dimension $n$ over $k$;
\item there exists a faithful representation $\rho : \g \rightarrow \gl_d(k)$;
\end{enumerate}
then for every $\chi \in \g^*$ there is a solvable weak polarisation $\s\subseteq \g$ such that $\rho(\s)$ is upper-triangularisable in $\gl_d(k)$.
\end{Proposition}

\begin{proof}
Fix $n, d\in \N$, $r \in \{0,...,n\}$ and let $k$ be an algebraically closed field of characteristic $p \geq 0$.
Let $\{v_1,...,v_n\}$ denote the standard basis for $k^n$. If $f$ is a bilinear map from $k^n\times k^n$ to $k^n$ then $f(v_i, v_j) = \sum_{l=1}^n f_{i,j; l} v_k$ and so we identify
the set of such bilinear maps with $k^{n^3}$ and identify $f$ with
$(f_{i,j; l})_{1\leq i,j,l\leq n} \in k^{n^3}$.
For $i=1,...,n$ we write $A_i$ for an element of $\Mat_d(k) \cong k^{d^2}$, so that the $n$-tuple
$A = (A_1,...,A_n)$ is an element
of $k^{nd^2}$. Finally write $\chi = (\chi_1,...,\chi_n) \in k^n$ and we view $\chi$ as an element of
$(k^n)^* = \Hom_k(k^n, k)$ via $v_i \mapsto \chi_i$.

Fix $r \in \{0,...,n\}$ and for any $(f, A, \chi) \in k^{n^3 + nd^2 + n}$ we consider the following 
four claims:
\begin{enumerate}
\item[(i)] $f = (f_{i,j;l})$ are the structure constants of a Lie bracket $[\cdot, \cdot]_f$ on $k^n$;
\item[(ii)] the Lie algebra $(k^n, [\cdot, \cdot]_f)$ has index equal to $r$;
\item[(iii)] the linear map $k^n \rightarrow \Mat_d(k)$ given by
$v_i \mapsto A_i$ is a faithful Lie algebra representation $\rho : k^n \rightarrow \gl_d(k)$;
\item[(iv)] There exist elements $x_1,...,x_s \in k^n$ where $s = \frac{1}{2}(n + r)$ which are linearly
independent and span a solvable Lie subalgebra $\s$ of $(k^n, [\cdot, \cdot]_f)$
such that $\chi([\s,\s]_f) = 0$ and
$\rho(\s)$ is upper-triangularisable  inside $\gl_d(k)$.
\end{enumerate}

Now consider the following statements indexed by $r$:
\begin{eqnarray}
\Phi_r & : & \forall (f, A, \chi) \in k^{n^3 + d^2 n + n} (\operatorname{(i)} \wedge \operatorname{(ii)}
 \wedge \operatorname{(iii)}) \Rightarrow \operatorname{(iv)})
\end{eqnarray}\vspace{-12pt}

\begin{center}
\begin{minipage}{0.8\textwidth}
\begin{enumerate}
\item[{\bf Claim:}] Each statement $\Phi_r$ can be formulated as a first-order sentence in the language $\Lr$ of rings (in the notation of \textsection\ref{ss:modeltheory}).  Moreover, the resulting sentences do not depend on the choice of $k$.
\end{enumerate}  
\end{minipage}
\end{center}

In Example~\ref{E:firstordersentences}(2) we
showed that (i) is given by a formula independent of $k$ in the language $\Lr$ with free variables $(f_{i,j;l})$.
If $(f_{i,j;l})$ are the structure constants
of a Lie bracket $[\cdot, \cdot]_f$ then the structure constants of the coadjoint representation
$\ad^*_f : k^n \rightarrow \Mat_n(k)$ are $(-f_{i,l;j})_{1\leq i,j,l\leq n}$.
It follows that for $x = (a_1,...,a_n) \in k^n$ the statement $\ad^*_f(x)\chi = 0$
can be expressed by the vanishing of certain polynomial functions, with integral coefficients, in the variables
$(f_{i,j;l})$ and $a_1,...,a_n, \chi_1,...,\chi_n$. The statement (ii) can be phrased in the following way:
\emph{there exists
$\psi = (\psi_1,...,\psi_n) \in k^n \cong \Hom_k(k^n, k)$ and linearly independent elements $x_1,...,x_r \in k^n$
that satisfy $\ad^*_f(x_i) \psi = 0$, and there does not exist
$\varphi \in \Hom_k(k^n, k)$ such that if $x_1,...,x_r \in k^n$ satisfy $\ad^*_f(x_i)\varphi = 0$
then $x_1,...,x_r$ are linearly dependent}.
This is a first-order formula which does not depend on the choice of $k$ in
$\Lr$ with free variables $(f_{i,j;l})$ thanks to part (1) of
Example~\ref{E:firstordersentences} and the previous remarks.
The fact that (iii) is a first-order formula which does not depend on the choice of $k$ in $\Lr$ with free variables $(f, A)$ follows similarly. Statement (iv) asserts the existence
of $x_1,...,x_s$ spanning a solvable Lie subalgebra of $(k^n, [\cdot, \cdot]_f)$.
The existence of elements that
satisfy $\chi([x_i, x_j]_f) = 0$ and span a solvable Lie algebra is a first-order formula not depending on $k$---indeed this follows quickly from Example~\ref{E:firstordersentences}(3). The fact that
the solvable subalgebra can be upper-triangularised in $\gl_d(k)$ can be expressed by the existence
of a basis of $k^d$ satisfying special properties which can also be expressed as first-order
formulas not depending on $k$ in $\Lr$. Since the last remark is proved in a manner almost identical to the previous parts,
we leave the details to the reader. The only free variables in (iv) are $(f, A, \chi)$ and so we have shown that
all of the variables in the formulas (i), (ii), (iii), (iv) are bound to quantifiers in $\Phi_r$.
Hence $\Phi_r$ is a first-order sentence in $\Lr$ which is independent of the choice of $k$, and this proves the claim.

Keep $n,d \in \N$, let $r \in \{0,...,n\}$ fixed, and now suppose that $k$ is algebraically closed of characteristic zero. By Lie's theorem \cite[Theorem~1.3.12]{Dix96}
we know that every solvable Lie subalgebra of $\gl_d(k)$ is 
upper-triangularisable, and thanks to \cite[Corollary~1.12.17]{Dix96}
we know that when $\g$ is a Lie algebra over $k$ of 
dimension $n$ and index $r$, then for all $\chi \in \g^*$ there exists a solvable subalgebra of $\g$ 
subordinate to $\chi$. Hence every algebraically closed field
of characteristic zero is a model for $\Phi_r$.
It follows by  the Lefschetz Principle
(Corollary~\ref{C:lefschetz}) that there is a number $p_1^r = p_1^r(n,d) \in \N$
such that $\Phi_r$ is also true when interpreted in any algebraically closed field of characteristic 
$p > p_1^r$. If we set $p_1 := \max\{p_1^0, p_1^1,...,p_1^n\}$ then it follows from the above 
remarks that for all $r = 0,...,n$, $\Phi_r$ is true
for every algebraically closed field of characteristic 
$p > p_1$. This completes the proof of the current Proposition.
\end{proof}

\begin{Lemma} \cite[Corollary~2.6]{KW71} \label{L:equidim}
If $\s$ is a solvable Lie algebra over $k$ and $\chi \in \g^*$ then all irreducible $U_\chi(\s)$-modules have the same dimension.
\end{Lemma}

\begin{Theorem}
\label{C:chchchchchangesagain}
For all $d \in \N$ there exists $p_0 = p_0(d) \in \N$ such that if $k = \overline{k}$ is a field of characteristic $p > p_0$, if $\g \subseteq \gl_d(k)$ and $M$ is a simple $\g$-module with $p$-character $\chi$ then $M$ is a quotient of a module of the form $\Ind_{\bar \s}^{\g, \chi} (M_0)$, where $\bar \s$ is a restricted Lie subalgebra of $\g$ which contains a solvable weak polarisation of $\chi$, and $M_0$ is a one-dimensional $U_\chi(\bar \s)$-module.
\end{Theorem}
\begin{proof}
Fix $d \in \N$, let $p_0 := \max\{d, p_1(1, d), p_1(2, d),...,p_1(d^2, d)\}$ where $p_1(n,d)$ was defined in Proposition~\ref{P:solvablepolarisationsexist}, and let $k$ be an algebraically closed field of characteristic $p > p_0$. Let $\g \subseteq \gl_d(k)$ be restricted and let $\chi \in \g^*$. Thanks to Proposition~\ref{P:solvablepolarisationsexist} we know there exists a solvable weak polarisation $\s \subseteq \g$ of $\chi$ with $\s$ upper-triangularisable. Let $\bar \s$ denote the restricted closure of $\s$ in $\gl_d(k)$, which is the smallest restricted Lie subalgebra of $\gl_d(k)$ containing $\s$. Evidently we have $\s \subseteq \bar\s \subseteq \g$. Since the upper-triangular matrices form a restricted Lie subalgebra of $\gl_d(k)$ it follows that $\bar \s$ is upper-triangularisable and solvable. Since $[\bar \s, \bar \s]$ consists of strictly upper-triangularisable matrices, our assumption $p > d$ ensures that the $p$-power map vanishes identically on $[\bar \s, \bar \s]$ and so this derived ideal is restricted.
By \cite[Proposition 2.1.3(1)]{SF88}, $\bar \s$ is spanned by elements of the form $x^{[p]^m}$ where $x \in \s$ and $m\geq 0$.  Since $[x^{[p]},y]= \ad(x^{[p]})(y)= \ad(x)^p(y)$ for all $x,y\in \g$, we see that $[\bar \s, \bar \s] = [\s, \s]$.

Now $[\bar \s, \bar \s]$ is a proper restricted
ideal of $\bar \s$, so $[\bar \s, \bar \s]$ generates a proper ideal of $U_\chi(\bs)$. The quotient is a nonzero finite-dimensional abelian $k$-algebra and therefore all of its simple modules are one-dimensional. This shows that $U_\chi(\bs)$ has at least one one-dimensional module.  We deduce from Lemma~\ref{L:equidim} that every simple $U_\chi(\bs)$-module is one-dimensional. Now let $M$ be a simple $\g$-module with $p$-character $\chi$. Then we can consider the socle of $M$ with respect to $U_\chi(\bs)$. According to our previous deductions there is a one-dimensional $U_\chi(\bs)$-submodule of this socle, which we call $M_0$. Now the existence of a surjection $\Ind_{\bar \s}^{\g,\chi}(M_0) \twoheadrightarrow M$ follows from the universal property of induced modules, along with the fact that $M$ is simple.
\end{proof}

Now the proof of Theorem~\ref{mainthm} follows from the previous theorem, combined with \eqref{e:induceddimn}.

\begin{Remark}
It has been conjectured that for a restricted Lie algebra $\g$ the following are equivalent:
\begin{enumerate}
\item[(i)] $\g$ is Frobenius, meaning $\ind(\g) = 0$;
\item[(ii)] there exists a non-empty open subset $\O \subseteq \g^*$ such that $U_\chi(\g)$ is simple for all $\chi \in \O$.
\end{enumerate}
Our main theorem implies that the conjecture holds for Lie subalgebra of $\gl_d(k)$ provided $\Char(k) > p_0(d)$. In fact (i) $\Rightarrow$ (ii) holds in general thanks to \cite[Theorem~4.4]{PS99}. Conversely, supposing $U_\chi(\g)$ is simple, there is a simple $\g$-module of the maximal possible dimension $\sqrt{\dim U_\chi(\g)} = p^{\frac{1}{2}\dim \g}$. If the KW1 conjecture holds for $\g$ then $M(\g) = p^{\frac{1}{2}\dim \g} = p^{\frac{1}{2}(\dim \g - \ind\g)}$ and so $\g$ is Frobenius.
\end{Remark}

\subsection{Example: the Lie algebras of group schemes}
\label{ss:groupschemes}


In this final subsection we draw attention to families of important Lie algebras to which Theorem~\ref{mainthm} can be applied, proving Theorem~\ref{thm:groupthm} from the introduction. We recall some of the elements of the theory of algebraic group schemes, following \cite{DG:1970}, \cite{Jan03}. Throughout the subsection we fix a commutative unital ring $R$, we write $R\alg$ for the category of $R$-algebras, and we say that $k$ is an {\it $R$-field} if $k$ is both a field and an object of $R\alg$. An affine algebraic group scheme $G$ over $R$ is a functor from $R$-algebras to groups, naturally equivalent to one of the form $\Spec_R R[G] := \Hom_{R\alg}(R[G], -)$ where $R[G]$ is some $R$-algebra. We call $R[G]$ the co-ordinate ring of $G$.  We say that $G$ is of finite type over $R$ if $R[G]$ is a finitely generated $R$-algebra.  The archetypal example of an affine algebraic group scheme is $\GL_d$.  Endowing an $R$-algebra $R[G]$ with the structure of an affine algebraic group scheme is equivalent to endowing $R[G]$ with a Hopf algebra structure $(R[G], \Delta, \sigma, \epsilon)$.  Associated to $G$ we have the Lie algebra $\g= \Lie(G)$ over $R$ \cite[II.4.1.2]{DG:1970}; we may identify $\g$ with the kernel of the homomorphism $G(R[\epsilon]/(\epsilon^2))\to G(R)$ induced by the $R$-algebra homomorphism $R[\epsilon]/(\epsilon^2)\to R, \epsilon\mapsto 0$.

When $G$ is an affine group scheme over $R$ and $k$ is any $R$-algebra then we can consider the base change $G_k$, which is an affine group scheme over $k$ obtained by viewing $k$-algebras as $R$-algebras. The co-ordinate ring $k[G_k]$ of $G_k$ is given by $k[G_k]= R[G]\otimes_R k$.  We write $\g_k = \Lie(G_k)$.  Since forming the Lie algebra commutes with base change \cite[II.4.1.4]{DG:1970}, we can identify $\g_k$ with $\g\otimes_R k$.  If $G$ is of finite type over $R$ then we can write $R[G] \cong R[x_1,...,x_n]/ I$ for some $n$ and some ideal $I$ of the polynomial ring $R[x_1,...,x_n]$; we obtain a map $\omega : R[x_1,...,x_n] \rightarrow k[x_1,...,x_n]$ and we have
\begin{eqnarray}
\label{e:presentingbasechage}
G_k \cong \Spec_k k[x_1,...,x_n]/ I_k
\end{eqnarray}
for some ideal $I_k$ of $k[x_1,...,x_n]$.

\begin{Lemma}
\label{L:representationdimd}
Let $G$ be an affine algebraic group scheme of finite type over $R$. There exists $d \in \N$ depending only on $G$ such that for each $R$-field $k$ there exists a representation $\rho : G_k \rightarrow (\GL_d)_k$, with $d\rho : \g_k \rightarrow (\gl_d)_k$ faithful.
\end{Lemma}
\begin{proof}
Suppose that $G$ corresponds to the Hopf algebra $(R[G], \Delta, \sigma, \epsilon)$, with $R[G] = R[x_1,...,x_n] / I$. Fix $i \in \{1,...,n\}$, write $\Delta(x_i) = \sum_{j=1}^{r(i)} f_{i, j}^{(1)} \otimes f_{i,j}^{(2)}$ and define $d := \sum_{i=1}^n r(i)$. Choose any $R$-field $k$ and let $M$ be the $R$-module generated by elements $\{f_{i,j}^{(1)} \mid i=1,...,n, j=1,...,r(i)\}$. Write $\omega : R[G] \to k[G_k]$ for the natural homomorphism.

Thanks to \eqref{e:presentingbasechage} there is a surjection $M \otimes_R k \twoheadrightarrow \omega(M)$ and so $\omega(M)$ identifies with a subspace of $k[G_k]$ of dimension $\leq d$. We observe that the coproduct $\Delta(\omega(x_i)) = \sum_{i=1}^{r(i)} \omega(f_{i,j}^{(1)}) \otimes \omega(f_{i,j}^{(2)})$  can be rewritten in the form $\Delta(\omega(x_i)) = \sum_{i=1}^{r_k(i)} h_{i,j} \otimes \omega(f_{i,j}^{(2)})$  for some $r_k(i) \le r(i)$ and certain elements $h_{i,j} \in \omega(M)$, such that $\omega(f_{i,1}^{(2)}),...,\omega(f_{i, r_k(i)}^{(2)})$ are $k$-linearly independent. According to \cite[I.2.13(4)]{Jan03} the space $N_i := \sum_{j=1}^{r_k(i)} k h_{i,j}$ is a $G_k$-submodule of $\omega(M)$ containing $x_i$. Furthermore it follows from \cite[Proof of Prop.~4.7]{Mil17} that $N = \sum_{i=1}^n N_i$ is a faithful $G_k$-submodule of $k[G_k]$ of dimension $\leq d$. Therefore $N \oplus k^{\oplus (d - \dim N)}$ is a faithful module of dimension $d$. Finally observe that $N$ is a faithful $\g_k$-module. To see this note, for example, that by faithfulness $1\to G(k[\epsilon]/(\epsilon^2))\to \GL_N(k[\epsilon]/(\epsilon^2))$ is exact; also $\Lie(G)\cong\ker (G(k[\epsilon]/(\epsilon^2))\to G(k))$ for the map $\epsilon\mapsto 0$, and similarly for $\GL_N$. Hence the claim follows from the commutativity of the following diagram \cite[II.4.1.3]{DG:1970}:
\[\begin{CD}0 @>>> \g_k @>>> G_k(k[\epsilon]/(\epsilon^2))\\
 @. @VVV @VVV \\
0 @>>> \gl_N @>>> \GL_N(k[\epsilon]/(\epsilon^2))
\end{CD}\qedhere\]
\end{proof}

\begin{Remark} Given a faithful representation $\sigma\colon G\to \GL_d$, we obtain by base change for any $R$-field $k$ a representation $\sigma_k\colon G_k\to (\GL_d)_k$ such that $d\sigma_k$ is faithful.  The question, however, of when such $\sigma$ exists is rather subtle---this is not known even when $R$ is the ring of dual numbers over a field $k$ and $G$ is flat and of finite type over $R$. (On the other hand, the existence of such a representation is known when $G$ is flat and of finite type over a Dedekind domain $R$, such as $\Z$, or indeed if $R$ is any field.)  Fortunately we don't need such a $\sigma$: any family of representations $\rho\colon G_k\to (\GL_d)_{k_p}$ for $p$ prime will do because we quantify over all representations simultaneously.
\end{Remark}


\begin{proof}[Proof of Theorem~\ref{thm:groupthm}]
 This follows immediately from Theorem~\ref{mainthm} and Lemma~\ref{L:representationdimd}.
\end{proof}
\begin{Example}
 Fix a Dynkin type and a non-negative integer $r$.  There exists a prime $p_0$ depending only on $r$ and the Dynkin type with the following property.  Suppose $G$ is a connected algebraic reductive group over an algebraically closed field $k$ of characteristic $p> 0$ such that $[G,G]$ has the given Dynkin type and $Z:= Z(G)^0$ has dimension $v$.  If $p> p_0$ then for any restricted Lie subalgebra $\h$ of $\g:= \Lie G$, the first Kac-Weisfeiler conjecture holds for $\h$.  To see this, note first that if $p$ is very good then $\Lie [G,G]$ has trivial centre, so
 any isogeny of connected reductive groups with domain or codomain equal to $G$ is separable and therefore gives rise to an isomorphism of restricted Lie algebras.  Hence without loss we can take $G$ to be of the form $[G,G]\times Z$, where $Z$ is a torus of dimension $r$ and $[G,G]$ is adjoint.  Then $G$ admits a faithful representation of dimension $\dim G+ r$ (just take the direct sum of the adjoint representation of $[G,G]$ and a faithful $r$-dimensional representation of $Z$).  This yields a faithful restricted representation of $\g$, and hence of $\h$.  The assertion now follows from Theorem~\ref{mainthm}.
 
 In particular, the result applies when $\h$ is a parabolic subalgebra of $\g$, or when $\h$ is the Lie algebra centraliser of some (possibly nonsmooth) subgroup scheme $H$ of $G$---note that in the latter case, $\h$ is the Lie algebra of the scheme-theoretic centraliser $C_G(H)$, so $\h$ is restricted.  In the special case $\h= \g$, it has been known for some time that the first Kac-Weisfeiler conjecture holds, at least when $p>2$ (see \cite[Section~4.1]{PS99}).
\end{Example}

\subsection{Example: families of non-algebraic Lie algebras}
\label{sec:notres}
Let $\g$ be a Lie algebra over a commutative unital ring $R$ such that $\g$ is finitely generated as an $R$-module.  If $k$ is an algebraically closed $R$-field then base change yields a finite-dimensional Lie algebra $\g_k$ over $k$.  It is natural to ask whether the equality (\ref{e:stateKW1}) from the first Kac-Weisfeiler conjecture holds when $\Char(k)\gg 0$.  In this final subsection we show that the answer is no; to see this we construct a Lie algebra $\g$ over the Gaussian integers $\Z[i]$ such that the Lie algebras $\g_p$ obtained by base change from $\Z[i]$ to $k_p$ exhibit some striking behaviour, quite different from the situation for Lie algebras of group schemes. (Here $k_p$ denotes the algebraic closure of the finite field ${\mathbb F}_p$.)  The underlying reason is that for infinitely many $p$, $\g_p$ does not admit a restricted structure.  This cannot happen in the setting of Theorem~\ref{thm:groupthm} because the Lie algebra of the affine group scheme $G_k$ is automatically restricted.

Let $\g$ be the free $\Z[i]$-module with basis $\{h,x,y\}$ and equip it with a $\Z[i]$-linear Lie algebra structure by defining brackets $[h,x] = x, [h,y] = iy, [x,y] = 0$. Since $\Z \subseteq \Z[i]$ is an integral extension, the going-up theorem implies that for every prime number $p \in \N$ there is a prime ideal $\p_p \in \Spec(\Z[i])$ such that $\Z[i]/\p_p$ has characteristic $p$. Since $\Z[i]$ is a principal ideal domain it follows that these quotients are fields, and we may write $k_p$ for the algebraic closure of $\Z[i]/\p_p$. Now define $\g_p := \g \otimes_{\Z[i]} k_p$. Note that $\g$ has a faithful representation in $\GL_4(\Z[i])$ given by
$$ h\mapsto \left(\begin{array}{cccc} 1&0&0&0\\0&-1&0&0\\0&0&i&0\\0&0&0&-i\end{array}\right),\ \ x\mapsto \left(\begin{array}{cccc} 0&1&0&0\\0&0&0&0\\0&0&0&0\\0&0&0&0\end{array}\right),\ \ y\mapsto \left(\begin{array}{cccc} 0&0&0&0\\0&0&0&0\\0&0&0&1\\0&0&0&0\end{array}\right), $$
and this gives rise by base change to a faithful representation of $\g_p$ in $\GL_4(k_p)$ for every prime $p$.  Since the example is elementary enough we can actually describe the representation theory explicitly.
\begin{Proposition}
\label{prop:notres}
Suppose that $p > 2$. The following are equivalent:
\begin{enumerate}
\item $M(\g_p) = p^{\frac{1}{2}(\dim(\g_p) - \ind(\g_p))}$;
\item $\g_p$ is a restricted Lie algebra;
\item $p \equiv 1$ modulo 4.
\end{enumerate}
\end{Proposition}
\begin{proof}
By Gauss' law of quadratic reciprocity we know that $i \in \mathbb{F}_p \subseteq k_p$ if and only if $p \equiv 1$ modulo 4. If this is the case then $\ad(h)^p = \ad(h)$ and it follows quickly that $\ad(h)^p, \ad(x)^p, \ad(y)^p \in \ad(\g_p)$. By Jacobson's theorem \cite[Theorem~2.2.3]{SF88} we deduce that $\g_p$ is restricted. Conversely, if $p \equiv 3$ modulo 4 then $i \notin \mathbb{F}_p$ and so  $\ad(h)^p \notin \ad(\g_p)$. Therefore by {\it loc.\ cit.}\ the algebra $\g_p$ is not restricted. Thus we see (2) $\Leftrightarrow$ (3).

Since $\dim(\g_p) - \ind(\g_p)$ is even we conclude that $\ind(\g_p) \in \{1,3\}$.
Clearly if $\chi[z, \g_p] = 0$ for all $z \in \g_p$ then $\chi(x)= \chi(y) = 0$, and so $\ind(\g_p) = 1$ for all $p > 0$.  Let $D_p$ denote the division ring of fractions of $U(\g_p)$ and let $Q_p$ denote the division ring of the centre $Z(\g_p)$ of $U(\g_p)$. Then $D_p$ is a $Q_p$-vector space and, thanks to \cite{Zas54}, we have $M(\g_p)^2 = [D_p : Q_p]$. Supposing $p \equiv 1$ modulo 4, so that $i \in \mathbb{F}_p$, we have central elements $h^p - h$ and $x^l y^m$ where $l, m \in \Z$ are any integers such that 
$l + im = 0 \in \mathbb{F}_p$. It follows that $U(\g_p)$ is generated as a $Z(\g_p)$-module by $\leq p^2$ elements. We deduce that $M(\g_p)^2 = [D_p : Q_p] \leq p^2$ and so $M(\g_p) \leq p$. By \cite[Remark~5.4, (1)]{PS99} it follows that $M(\g_p) = p$, since $\g_p$ is restricted. From $\frac{1}{2}(\dim(\g_p) - \ind(\g_p)) = 1$ it follows that (3) $\Rightarrow$ (1). Now suppose that $p \equiv 3$ modulo 4. A very explicit calculation shows that $Z(\g_p)$ is a polynomial algebra generated by $\{x^p, y^p, h^{p^2} - h\}$ and so $U(\g_p)$ is a free $Z(\g_p)$-module of rank $p^4$. It follows that $M(\g_p)^2 = [D_p : Q_p] = p^4$ and so $M(\g_p) \neq p^{\frac{1}{2}(\dim(\g_p) - \ind(\g_p))} = p$ in this case, whence (1) $\Rightarrow$ (3). This concludes the proof.
\end{proof}

\section*{Appendix by Akaki Tikaradze}

 In this appendix a short, alternative proof of Theorem~\ref{thm:groupthm} is presented, which applies to group schemes defined over a finitely generated ring $R \subseteq \C$. As was noted previously, in order to prove the KW1 conjecture for a restricted Lie algebra $\g$ it suffices to demonstrate that $M(\g) \le p^{\frac{1}{2}(\dim \g - \ind \g)}$. Thanks to the work of Zassenhaus this is equivalent to showing that the dimension of the skew division ring $D(\g)$ of fractions of $U(\g)$ is a field extension of rank $\le p^{\dim \g - \ind \g}$ over its centre (see Lemma~\ref{equivalentforms} for more detail); this shall be proven by combining Rosenlicht's theorem with reduction modulo $p$ and deformation arguments. As a by-product of the proof, a description of the centre of $D(\g)$ is obtained which confirms a slight modification of a conjecture of Kac \cite{KPre}.
 
Given a commutative (noncommutative Noetherian) domain $R$, denote by $\Frac(R)$ its quotient field (skew field) of fractions. If $R$ is commutative and $M$ is any $R$-module and $S$ is an $R$-algebra write $M_S := M\otimes_R S$ and, when $R$ is also an integral domain and $M$ is finitely generated, we refer to $\dim_{\Frac(R)} M_{\Frac(R)}$ as {\em the rank of $M$}.

Continue to write $k$ for a field of characteristic $p > 0$. Given a restricted Lie algebra $\mathfrak{g}$ over $k$ recall that $Z_p(\g)$ denotes the $p$-centre of $\g$. Furthermore, whenever $G$ is an $R$-group scheme and $k$ is an $R$-field write $\g_k$ for the Lie algebra of the scheme $G_k$ obtained by base change. 
\begin{Theorem}\label{main}
Let $R\subset\mathbb{C}$ be a finitely generated ring, let $G$ be an $R$-group scheme and $R \to k$ a base change to an algebraically closed field. Then provided $\Char(k) = p \gg 0$ the KW1 conjecture holds for $\g_k.$

\end{Theorem}

As a consequence of the proof of Theorem \ref{main}, we also show the following.
\begin{Theorem}\label{center}
Let $G$ be a group scheme over a finitely generated ring $R \subseteq \C$.
Then provided $\Char(k) = p \gg 0$ the center of $D(\g_k)$ is generated as a field by $Z_p(\mathfrak{g}_k)$ along with central elements obtained by reduction modulo $p$.
\end{Theorem}

For the proof we will need to recall couple of simple lemmas from commutative algebra.
\begin{Lemma}\label{rank}
Let $R\subset\mathbb{C}$ be a finitely generated ring.
Let $A$ be a finitely generated commutative algebra over $R$ such that $A_{\mathbb{C}}$ is a domain.
 Let $B\subset A$ be a finitely generated  $R$-subalgebra.
Then for all $p\gg 0$ and any base change $R\to k$ to an algebraically closed field of characteristic $p$
the rank of $A_k$ over $B_kA_k^p$ is  $p^{\dim(A)-\dim(B)},$ where $\dim(A)$ and $\dim(B)$ denote the Krull dimension of $A$ and $B$ respectively.
\end{Lemma}
\begin{proof}
By localizing $B$ if necessary, we may assume by the Noether  normalization lemma  that
$A$ is a finite module over $B[x_1,..., x_n]$, where $ x_1,...x_n\in A$ are algebraically independent over $B.$
Let $m$ be a number of generators of $A$ as a module over $B[x_1,..., x_n].$ Since $A$ if a finite extension of $B[x_1,...,x_n]$ we have $n=\dim(A)-\dim(B).$
Now let $S\to k$ be a base change to an algebraically closed field of characteristic $p\gg 0$ (in particular $p>m$.)
So $A_k$ is a finite  module with at most $m$ generators over $B_k[x_1,\cdots, x_n].$
Clearly the rank of $A_k$ over $B_k[x_1^p,..., x_n^p]\subset B_kA_k^p$
is $p^nl,$ where $l\leq m$ is the rank of $A_k$ over $B_k[x_1,\cdots, x_n].$
On the other hand, it can easily be seen that if $K$ is a field of characteristic $p > 0$ which is finitely generated over a perfect subfield, then $ K / K^p$ is a finite extension and $[K : K^p]$ is a power of $p$ (by picking a transcendence basis over the perfect subfield we can reduce to the case of finite field extensions, and then use the fact that every finite extension of a perfect field is perfect). It follows that the rank $A_k$ over $A_k^p$ is a power of $p.$
So, the rank of $A_k$ over $B_kA_k^p$ is a power of $p$ and divides $p^nl$,
hence it must divide $p^n.$ However, the rank of $B_kA_k^p$ over $B_k[x_1^p,..., x_n^p]$
is at most $m.$ Thus the rank of $A_k$ over $B_kA_k^p$ is  $p^{\dim(A)-\dim(B)}.$

\end{proof}

\begin{Lemma}\label{Gr}
Let $R$ be a nonnegatively filtered commutative  algebra over a field $k$ such that 
$\gr(R)$ is a finitely generated $k$-domain.
Let $M$ be a filtered $R$-module such that $\gr(M)$ is a finitely generated
$\gr(R)$-module. Then the rank of $M$ over $R$ equals to the rank of $\gr(M)$ over $\gr(R).$
\end{Lemma}
For a proof of Lemma \ref{Gr} see for example \cite[Lemma 2.3]{Tik11}. For a more general
result covering noncommutative algebras see \cite[Lemma 6.2]{Gor07}. 
We record one final useful lemma which is needed to complete the proof of Theorems~\ref{main} and \ref{center}.
\begin{Lemma}
\label{equivalentforms}
If $\g$ is a restricted Lie algebra over an algebraically closed field $k$, then the following are equivalent:
\begin{enumerate}
\item[(i)] the first Kac--Weisfeiler conjecture holds for $\g$;
\item[(ii)] the rank of $U(\g)$ over $Z(\g)$ is at most $p^{\dim \g - \ind \g}$;
\item[(iii)] the rank of $Z(\g)$ over $Z_p(\g)$ is at least $p^{\ind(\g)}$.
\end{enumerate}
\end{Lemma}
\begin{proof}
It was demonstrated in \cite[Theorem~6]{Zas54} that
$M(\g)^2$ equals the rank of $U(\g)$ over $\Frac Z(\g)$,
hence (i) holds if and only if we have equality in (ii). The inequality $M(\g)^2 \geq p^{\dim \g - \ind \g}$ was deduced in \cite[Remark~5.4]{PS99} and so (i) $\Leftrightarrow$ (ii). The equivalence (ii) $\Leftrightarrow$ (iii) follows from the fact that the rank of $U(\g)$ over $Z_p(\g)$ is $p^{\dim \g}$.
\end{proof}

\begin{proof}[Proof of Theorems \ref{main} and \ref{center}]
Let $G$ be the connected complex algebraic group corresponding to $\mathfrak{g}_\C.$
Let $m$ be the index of $\mathfrak{g}_\C.$
By Rosenlicht's theorem $\mathbb{C}(\mathfrak{g}_\C^*)^G$ has the transcendence degree $m$ over $\mathbb{C}.$
Let $f_i, g_i\in \mathbb{C}[\mathfrak{g}^*_{\mathbb{C}}]$ be elements such that
$$\left\{\frac{f_i}{g_i}\in \mathbb{C}(\mathfrak{g}_{\mathbb{C}}^*)^G\mid  1\leq i\leq m\right\}$$
are algebraically independent elements,
written as reduced fractions. We write $S_R(\g)$ for the symmetric $R$-algebra generated by the $R$-module $\g$, and identify $\C(\g_\C^*)$ with $\Frac S(\g_\C)$.
By localizing $R$ if necessary, we may assume that $f_i, g_i\in S_R(\g), 1\leq i\leq m.$
Denote by $\phi_i$ (respectively $\psi_i$) the image of $f_i$ (respectively $g_i$)
under the symmetrization map $S(\g_\C)\to U(\g_\C)$. Localizing $R$ further if necessary we can assume that $\phi_i, \psi_i$ are elements of the enveloping $R$-algebra $U_R(\g)$ (the $R$-tensor algebra of $\g$ modulo the usual relations). It is well-known that $\phi_i, \psi_i$ are semi-invariants for the action of $\g_\C$ of the same weight. After localising $R$ further we deduce that there is an $R$-linear map $\chi_i : \g \to R$ such that $\ad(x) \phi_i = \chi_i(x) \phi_i$ and $\ad(x) \psi_i = \chi_i(x) \psi_i$ for all $x \in \g$. Now if $R \to k$ is a base change to any algebraically closed field then it follows that the quotient $\frac{\phi_i \otimes 1}{\psi_i\otimes 1} \in \Frac U(\g_k)$ is central. Furthermore it follows from elementary linear algebra that $m=\ind(\mathfrak{g}_k)$ provided $\Char(k) = p \gg 0$.

Thanks to Lemma~\ref{equivalentforms}, to establish the KW1 conjecture for $\mathfrak{g}_k$ we need to show that the rank of  $U(\mathfrak{g}_k)$ over
$Z(U(\mathfrak{g}_k))$ is at most $p^{\dim\mathfrak{g}_k-m}.$
For this purpose, it is enough to show that the  rank of $S(\g_k)$ as a $\gr(Z(U(\mathfrak{g}_k))$-module is at most $p^{\dim\mathfrak{g}_k-m}$ thanks to Lemma~\ref{Gr}.
Let $z_i\in U(\mathfrak{g}_k)$ be such that $\psi_iz_i\in Z(U(\mathfrak{g}_k)).$ So $\phi_iz_i\in Z(U(\mathfrak{g}_k)).$
Now since $$\gr(\phi_i)\gr(z_i), \gr(\psi_i)\gr(z_i)\in \gr Z(U(\mathfrak{g}_k)),$$
it follows that 
\begin{eqnarray}
\label{e:fracs}
\frac{f_i \otimes 1}{g_i \otimes 1}\in \Frac( \gr Z(U(\mathfrak{g}_k))), 1\leq i\leq m.
\end{eqnarray}
If we put $A = S_R(\g)[g_1^{-1},..., g_m^{-1}]$ and $B = R[\frac{f_1}{g_1}, \cdots, \frac{f_m}{g_m}]$, then for any base change $R \to k$ we may view $A_k$ and $B_k$ as subalgebras of $\Frac S(\g_k)$. Since $\psi_i \otimes 1$ is a semi-invariant for $\ad(\g_k)$ it follows that $\psi_i^p \otimes 1 \in U(\g) \otimes_R k \cong U(\g_k)$ is actually central for $i =1,...,p$. Therefore $g_i^p \otimes 1 = \gr \psi_i^p \otimes 1 \in \gr Z(U(\g_k))$. Combining this with \eqref{e:fracs} it follows that
$\Frac(\gr Z(U(\mathfrak{g}_k))$)  contains  $A_k^pB_k$,
and so it suffices to show that the rank of $A_k$ over $A_k^p B_k$ is at most $p^{\dim\mathfrak{g}_k-m},$
which follows from  Lemma~\ref{rank}. This completes the proof of Theorem~\ref{main}.

Applying Lemma~\ref{equivalentforms}, we see that the rank of $Z(\g_k)$ over $Z_p(\g_k)$ is $p^m$ for $\Char k \gg 0$. Let $Q$ denote the subfield of $\Frac Z(\g_k)$ generated by $\Frac Z_p(\g_k)$ and the elements
$$\frac{\phi_i\otimes 1}{\psi_i \otimes 1} \text{ for } i=1,...,m.$$
The above argument actually shows that $[D(\g)  :Q] = p^{\dim \g - m}$, which implies $[Q: \Frac Z_p(\g)] = p^m$. It follows that the inclusion $Q \subseteq \Frac Z(\g_k)$ must be an equality. This completes the proof of Theorem~\ref{center}.

\end{proof}
\begin{Remark}
Let $\mathfrak{g}$ be a Lie algebra of a $\Z$-group scheme and $k$ be a field of large characteristic $p > 0$. Then it is natural to ask whether the center of $U(\mathfrak{g}_k)$ is generated by $Z_p(\mathfrak{g}_k)$ and the mod $p$ reduction of the center of $U(\g)$
(this was conjectured by Kac in \cite{KPre}). Unfortunately this is not always true: Let $\mathfrak{g}$ be a three dimensional free $\Z$-module spanned by $h, x, y$ with Lie brackets given by
$$[h, x]=nx, [h, y]=my, [x, y]=0, n, m\in\mathbb{Z}_{>0}, (n, m)=1.$$
Then the center of $U(\mathfrak{g}_k)$ is generated by $h^p-h, x^p, y^p, x^iy^j$
where $ ni+mj=0$ \ \ mod $p$ and $0<i, j<p.$ On the other hand the center of $U(\g)$ is trivial, hence $Z(U(\mathfrak{g}_k))\neq Z_p(\mathfrak{g}_k).$
Meanwhile, the center of $D(\mathfrak{g}_\C)$ is generated by the elements $x^my^{-n}$
which are also generators for the centre of $D(\mathfrak{g}_k)$ over $Z_p(\mathfrak{g}_k)$.
\end{Remark}

{\small
\bibliographystyle{amsalpha}
\bibliography{bib}}

\providecommand{\bysame}{\leavevmode\hbox to3em{\hrulefill}\thinspace}
\providecommand{\MR}{\relax\ifhmode\unskip\space\fi MR }
\providecommand{\MRhref}[2]{%
  \href{http://www.ams.org/mathscinet-getitem?mr=#1}{#2}
}
\providecommand{\href}[2]{#2}
\begin{thebibliography}{Mar02}

\bibitem[DG70]{DG:1970}
M.~Demazure and P.~Gabriel, \emph{Groupes alg\'ebriques. {T}ome {I}:
  {G}\'eom\'etrie alg\'ebrique, g\'en\'eralit\'es, groupes commutatifs}, Masson
  \& Cie, \'Editeur, Paris, 1970, Avec un appendice {\em Corps de classes
  local} par M. Hazewinkel.

\bibitem[Dix96]{Dix96}
J.~Dixmier, \emph{Enveloping algebras}, Graduate Studies in Mathematics,
  vol.~11, American Mathematical Society, Providence, RI, 1996, Revised reprint
  of the 1977 translation. \MR{1393197 (97c:17010)}

\bibitem[Gor07]{Gor07}
I.~Gordon, \emph{Gelfand-kirillov conjecture for symplectic reflection
  algebras}, arxiv:0710.1419v2.pdf (2007).

\bibitem[Jac37]{Jac37}
Nathan Jacobson, \emph{Abstract derivation and {L}ie algebras}, Trans. Amer.
  Math. Soc. \textbf{42} (1937), no.~2, 206--224. \MR{1501922}

\bibitem[Jan98]{Jan98}
Jens~Carsten Jantzen, \emph{Representations of {L}ie algebras in prime
  characteristic}, Representation theories and algebraic geometry ({M}ontreal,
  {PQ}, 1997), NATO Adv. Sci. Inst. Ser. C Math. Phys. Sci., vol. 514, Kluwer
  Acad. Publ., Dordrecht, 1998, Notes by Iain Gordon, pp.~185--235.
  \MR{1649627}

\bibitem[Jan03]{Jan03}
J.~C. Jantzen, \emph{Representations of algebraic groups}, second ed.,
  Mathematical Surveys and Monographs, vol. 107, American Mathematical Society,
  Providence, RI, 2003. \MR{MR2015057 (2004h:20061)}

\bibitem[Kac95]{KPre}
V.~Kac, \emph{Featured {M}ath. review of \emph{Irreducible representations of
  reductive Lie algebras of reductive groups and the {K}ac--{W}eisfeiler
  conjecture} by {A}.~{P}remet.}, Invent. Math. \textbf{121} (1995), no.~1,
  79--117.

\bibitem[Mar02]{Ma02}
David Marker, \emph{Model theory}, Graduate Texts in Mathematics, vol. 217,
  Springer-Verlag, New York, 2002, An introduction. \MR{1924282}

\bibitem[Mil17]{Mil17}
J.~S. Milne, \emph{Algebraic groups}, Cambridge Studies in Advanced
  Mathematics, vol. 170, Cambridge University Press, Cambridge, 2017, The
  theory of group schemes of finite type over a field. \MR{3729270}

\bibitem[PS99]{PS99}
Alexander Premet and Serge Skryabin, \emph{Representations of restricted {L}ie
  algebras and families of associative {$\mathcal L$}-algebras}, J. Reine
  Angew. Math. \textbf{507} (1999), 189--218. \MR{1670211}

\bibitem[SF88]{SF88}
H.~Strade and R.~Farnsteiner, \emph{Modular {L}ie algebras and their
  representations}, Monographs and Textbooks in Pure and Applied Mathematics,
  vol. 116, Marcel Dekker Inc., New York, 1988. \MR{929682 (89h:17021)}

\bibitem[Tik11]{Tik11}
Akaki Tikaradze, \emph{On the {A}zumaya locus of almost commutative algebras},
  Proc. Amer. Math. Soc. \textbf{139} (2011), no.~6, 1955--1960. \MR{2775371}

\bibitem[Top17]{To17}
Lewis Topley, \emph{A non-restricted counterexample to the first
  {K}ac-{W}eisfeiler conjecture}, Proc. Amer. Math. Soc. \textbf{145} (2017),
  no.~5, 1937--1942. \MR{3611310}

\bibitem[VK71]{KW71}
B.~Ju. Ve\u{\i}sfe\u{\i}ler and V.~G. Kac, \emph{The irreducible
  representations of {L}ie {$p$}-algebras}, Funkcional. Anal. i Prilo\v zen.
  \textbf{5} (1971), no.~2, 28--36. \MR{0285575}

\bibitem[Zas54]{Zas54}
Hans Zassenhaus, \emph{The representations of {L}ie algebras of prime
  characteristic}, Proc. Glasgow Math. Assoc. \textbf{2} (1954), 1--36.
  \MR{0063359}

\end{thebibliography}
\end{document}